\documentclass[12pt]{article}
\usepackage{amsmath,amssymb,amsfonts,amsthm,url,xspace,graphicx,color,mathrsfs,enumitem}
\usepackage{tikz}
\usepackage{bm}
\usepackage[colorlinks=true,linkcolor=blue,citecolor=blue,urlcolor=blue,pdfborder={0 0 0}]{hyperref}
\usepackage{graphicx}

\usepackage{xcolor}

\theoremstyle{plain}
\newtheorem{thm}{Theorem}[section]

\newtheorem{lemma}[thm]{Lemma}

\newcommand{\eps}{\varepsilon}

\newcommand{\Z}{\mathbb{Z}}

\newcommand{\E}{\mathcal{E}}
\newcommand{\R}{\mathbb{R}}

\newcommand{\ENT}{\mathrm{Ent}}
\newcommand{\Li}{\mathrm{Li}}

\newcommand{\be}{\begin{equation}}
\newcommand{\ee}{\end{equation}}

\newcommand{\old}[1]{}

\title{Patterns in sequences}
\author{Richard Kenyon}

\begin{document}
\maketitle
\abstract{We study pattern densities in binary sequences, finding optimal limit sequences with 
fixed pattern densities.}

\section{Introduction}

Given two finite binary words $w\in\{0,1\}^m$ and $X\in\{0,1\}^n$, where $m\le n$, we let $N_w(X)$ 
be the number of times $w$ occurs as a not-necessarily-consecutive subword of $X$. 
For example $N_{10}(0100101) = 4$.
We define $\rho_w(X) := N_w(X)/\binom{n}{m},$ the \emph{density of pattern $w$ in $X$}. 
It is the probability that a random length-$m$ subsequence of $X$ is $w$. 

What is the maximum possible density $\rho_w$ for a sequence $X$ of length $n$?
What does a typical sequence $X$ with a given density $\rho_w$ look like?

Such pattern counting questions occur in several different areas of combinatorics.
For binary sequences, the simplest example, of $10$ patterns, corresponds to integer partitions:
representing $X$ as a Young diagram, its area is $N_{10}(X)$ (see Figure \ref{youngdiagram}).
\begin{figure}[htbp]
\begin{center}
\includegraphics[width=3cm]{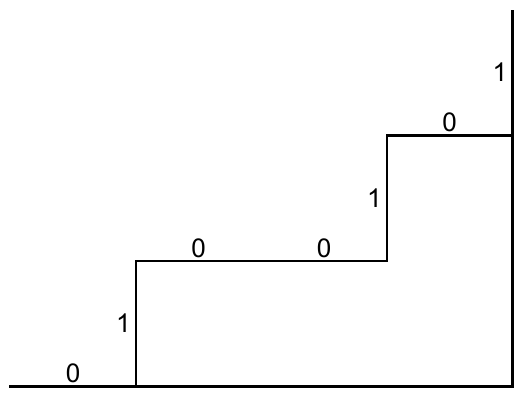}
\caption{\label{youngdiagram}A binary word can be presented as the boundary of a Young diagram: a NE lattice path where $0$ is a step east and $1$ a step north. Then $N_{10}$ is the area under the Young diagram. Here $N_{10}(0100101)=4$.}
\end{center}
\end{figure} 
Generally there are connections between pattern counting and paths in various integer Heisenberg groups. For the standard Heisenberg group generated by
$$M_0=\begin{pmatrix}1&1&0\\0&1&0\\0&0&1\end{pmatrix}~~~~M_1=\begin{pmatrix}1&0&0\\0&1&1\\0&0&1\end{pmatrix},$$
the upper right entry of a positive product is $N_{01}$. For example the upper right entry of $M_0M_1M_0M_1M_1$ is $N_{01}(01011)$. 
See Section \ref{positivity} below.
Binary pattern counting questions also turn out to be related to certain extremal probability measures on $[0,1]$, 
for example the probability measure
on $[0,1]$ that maximizes $x(y-x)$, where $x,y$ are the smaller and larger of two independent samples; see below.

Analogous pattern density questions were asked, and partially answered, in the setting of graphs with subgraph densities
in \cite{Lovasz, CD, KRRS} and in the setting of permutations and pattern densities in \cite{HKMRS, KKRW}. See also \cite{CR}
for a general framework.

One concrete instance of binary pattern counting is the ``BRBR game'':
From a standard deck of $52$ playing cards, a hand of four cards is drawn uniformly at random
from the $\binom{52}{4}$ possibilities,
keeping the cards in the same order as they were in the deck. The goal is to arrange the original deck so as to maximize the probability of 
getting BRBR, that is, first and third card black, second and fourth card red.  If the deck is randomly shuffled, the probability is 
about $1/16$ of winning. If the deck is in ``new deck order'', that is, consists in $13$ blacks followed by $13$ reds followed by $13$ blacks followed by $13$ reds, the probability is almost $70\%$ higher: about $0.1055$. How does one arrange the deck so as to maximize
the winning probability? 

The BRBR game is a special case of pattern counting: optimizing patterns $1010$.
Interestingly, one can do better than ``new deck order''; we show below that the maximum probability tends to $\frac{3}{4e^2}$ for large 
decks (which are half red and half black),
and we characterize the corresponding optimal arrangement. (For $52$-card decks the apparently optimal arrangement has winning probability $\approx0.114$, see Figure \ref{BRBR52opt}.)
\begin{figure}[htbp]
\begin{center}
\includegraphics[width=4cm]{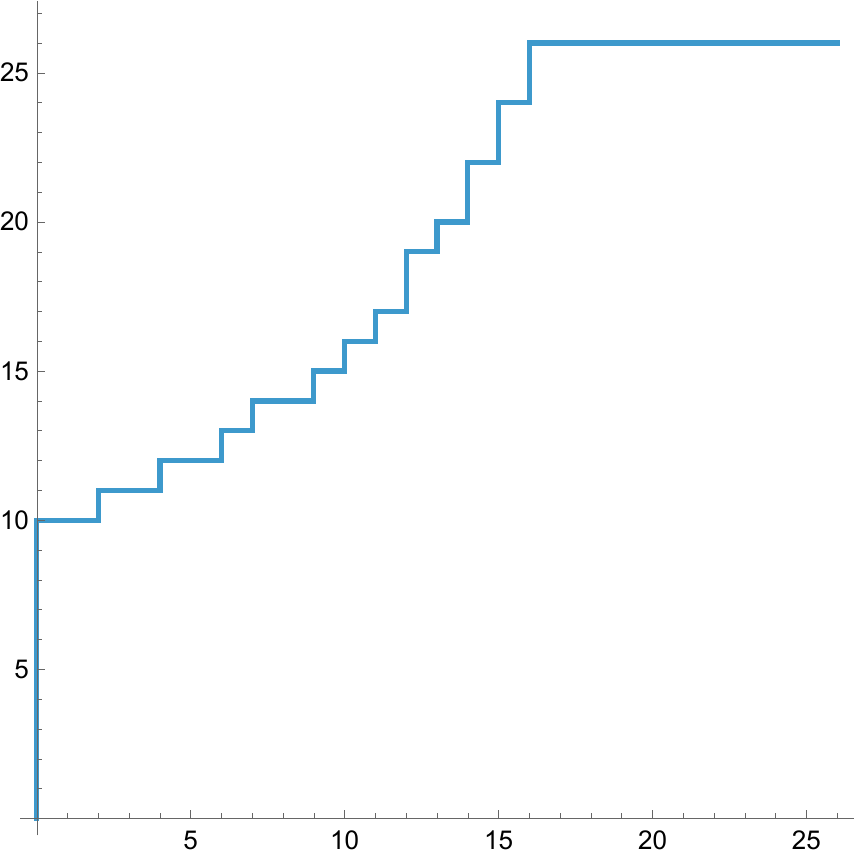}
\caption{\label{BRBR52opt}Conjecturally optimal arrangement of a $52$ card deck for the BRBR game.
Starting from the origin, a north step is a black card; an east step is a red card. This arrangement gives $p\approx0.114$. Compare with Figure \ref{F1010YD}.}
\end{center}
\end{figure} 
This example illustrates a surprising phenomenon, which is that limit shape patterns can sometimes be \emph{nonanalytic}: see Figures \ref{f1010} and \ref{F1010YD}.
This is an unusual feature of the pattern counting model, and makes the analysis challenging.
\bigskip
 
In this paper we consider the following two general pattern-counting problems. 
\begin{enumerate}
\item For a given set of patterns $w_1,\dots,w_k$, what is the ``feasible set" of allowed tuples of densities 
$(\rho_{w_1}(X),\dots,\rho_{w_k}(X))$, for $X$ of large length?
\item For large $n$, what does a typical sequence of length $n$, and having given pattern densities $\rho_{w_i}=\delta_i$,
resemble? 
\end{enumerate}

We answer the first question in the case of two patterns, $w_1,w_2$, where $w_1=1$ and $w_2=w$ is arbitrary. We answer the second question
for some sets of short patterns, and answer both questions when all patterns are of type $w_k=1^k0$. 
We also discuss the algebraic relations among pattern densities, and state a conjecture 
about the number of algebraically independent pattern densities.

The appropriate limit objects for pattern counting are called \emph{sublebesgue measures}.
A sublebesgue measure is a measure $f(x)\,dx$ on $[0,1]$ which is absolutely continuous with respect to 
Lebesgue measure and is of density bounded by $1$: $f(x)\in [0,1]$. 
We let $\Omega$ denote the set of sublebesgue measures. 
A binary sequence $X=(x_1,\dots,x_n)$ has an associated sublebesgue measure $\mu_X =f_X(x)dx\in\Omega$ where $f_X$ is the $\{0,1\}$-valued step function whose value on $[\frac{i-1}n,\frac{i}n)$ is $x_i$.

Suppose we have a sequence $\{X_n\}_{n=1,2,\dots}$ with $X_n\in\{0,1\}^n$, for which the corresponding measures $\mu_{X_n}$ converge weakly
to a sublebesgue measure $\mu(x) = f(x)dx$.
In this case for each pattern $w$ the limiting density $\lim_{n\to\infty} \rho_w(X_n)$ of $w$ exists and is a function solely of $\mu$. 
For example 
$$\rho_1(\mu) = \int_0^1f(x)\,dx = F(1)$$
where $F(y)=\mu([0,y])=\int_0^y f(x)\,dx$ is the distribution function of $\mu$.
As another example
$$\rho_{10}(\mu) = 2\iint_{0\le x<y\le 1}f(x)(1-f(y))\,dx\,dy$$
and generally for a word $w$ of length $k$, 
\be\label{genpattern}\rho_w = k!\int_{0\le x_1<\dots<x_k\le1}g(x_1)g(x_2)\dots g(x_k)dx_1\,\dots\,dx_k\ee
where $$g(x_i) =\begin{cases} f(x_i) & w_i=1\\ 1-f(x_i)&w_i=0.\end{cases}$$

This paper is organized as follows.
In Section \ref{Wass} we prove that pattern densities characterize sublebesgue measures, and conversely,
so in essence sublebesgue measures are the appropriate limit objects to discuss pattern densities,
in the same sense that graphons are the appropriate limit objects to discuss subgraph densities.
In Section \ref{algebraic} we discuss the various algebraic relations between pattern densities,
finding all relations for patterns up to length $5$, and conjecturing about the number of independent patterns
of general length.
In Section \ref{feasible} we discuss feasibility regions, and give some examples. We prove (Theorem \ref{Cthm})  that 
the feasibility region $E_{1,\tau}$ for two patterns, $\rho_1$ and $\rho_\tau$, is of the form 
$$E_{1,\tau} = \{(\rho_1,\rho_\tau)~:~\rho_1\in[0,1]~~\text{and}~~0\le\rho_\tau\le C_\tau\rho_1^k(1-\rho_1)^l\}$$
for a constant $C_\tau$ independent of $\rho_1$. Here $k,l$ are the number of $1$s and $0$s, respectively, in $\tau$.
We also analyze the BRBR game in this section.
In Section \ref{variational} we discuss limit shapes of sequences with fixed pattern densities. In Section \ref{1k0} we discuss the
case of patterns $1^k0$, finding explicit limit shapes. 
Finally in Section \ref{positivity} we discuss the relation with upper triangular matrices.

\bigskip
\noindent{\bf Acknowledgments.} This work was partially supported by the Simons foundation grant 327929. We thanks Pete Winkler for discussions which initiated this work.
We thank Pavel Galashin for discussions and pointing out the relation with total
positivity. We thank Mei Yin and Persi Diaconis for comments and suggestions.

\section{Pattern densities and sublebesgue measures}\label{Wass}

The natural notion of distance on the space of sublebesgue measures is the 
Wasserstein distance, or equivalently the $L^1$ metric on the distribution functions:
$$d_W(\mu_1,\mu_2) = \int_0^1|F_1(x)-F_2(x)|\,dx$$
where $F_i(x) = \mu_i([0,x])$ is the distribution function of $\mu_i$. The corresponding 
metric topology is also the topology of weak convergence of the measures. 

\begin{thm} Two sublebesgue measures are a.e.\! equal if and only if they have the same pattern densities. More generally,
a sequence of sublebesgue measures converges in metric $d_W$ if and only if its set of pattern densities converge. 
\end{thm}

\begin{proof} From (\ref{genpattern}) each pattern density is an integral of the density $f$, and so if a sequence of densities converge weakly then the pattern densities converge. 
Conversely, the Hausdorff Moment Theorem \cite{Hausdorff} states that a sublebesgue measure $f(x)dx$ is determined by its moments
$m_n = \int x^n f(x)\,dx$, and, using (\ref{genpattern}) again, each moment $m_n$ is a finite sum of pattern densities
$m_n = \frac1{n+1}\sum_{w\in\{0,1\}^n}\rho_{w1}$. 
\end{proof}

\section{Algebraic relations}\label{algebraic}
There are many algebraic relations among the $N_\tau$'s.
For strings $\tau$ of length $1$, for example, we have, for a string $X$ of length $n$,
$$N_0(X)+N_1(X)=n.$$
For $\tau$ of length $2$ we have the relations
\begin{align*}
N_{00}&=\binom{N_0}{2}\\
N_{01}+N_{10}&=N_0N_1\\
N_{11}&=\binom{N_1}{2}.
\end{align*}
The middle equation, for example, says that for every choice of a $0$ and a $1$ in $X$, either the $0$ precedes the $1$ or the $1$ precedes the $0$. 
These four equations have the consequence that, given $n$, among strings $\tau$ of length $\le2$ there are only two algebraically independent pattern-counting quantities, for example $\{N_1,N_{10}\}$. 

For patterns of length $3$ we have similar linear relations
\begin{align}
\nonumber N_{000}&=\binom{N_0}{3}\\
\label{3linear}N_{001}+N_{010}+N_{100}&=\binom{N_0}2N_1\\
\nonumber N_{110}+N_{101}+N_{011}&=N_0\binom{N_1}2\\
\nonumber N_{111}&=\binom{N_1}{3}.
\end{align}
There are four other relations
\begin{align}
\label{N0N10}N_0N_{10}&= N_{010} + 2N_{100}+N_{10}\\
N_0N_{01}&= N_{010}+2N_{001}+N_{01}
\end{align}
and their complements (switching $0$ and $1$). The equation (\ref{N0N10}), for example, is obtained by taking an instance of a $10$ pattern in $X$ and a $0$ in $X$ and considering the $4$ ways the $0$ can be in relation to the $10$ pattern: to its left, between the two,
to its right, or overtop the $0$. 

Beyond the strings of length $1$ and $2$ these $8$ relations yield $2$ more algebraically independent quantities, for example
$N_{100}$ and $N_{110}$. To see that $N_1,N_{10}, N_{100}$ and $N_{110}$ are algebraically independent,
note that the strings $01110001$ and $10100110$ have the same $N_1,N_{10}$ and $N_{100}$ but different $N_{110}$.

Going further, for length $4$ we have the $5$ linear relations analogous to (\ref{3linear})
$$\sum_{x_i\in\{0,1\}: \sum x_i=k} N_{x_1x_2x_3x_4}= \binom{N_0}{4-k}\binom{N_1}{k}.$$
\old{
\begin{align*}
N_{0000}&=\binom{N_0}{4}\\
N_{0001}+N_{0010}+N_{0100}+N_{1000}&=\binom{N_0}{3}N_1\\
N_{0011}+N_{0101}+N_{0110}+N_{1001}+N_{1010}+N_{1100}&=\binom{N_0}2\binom{N_1}2\\
N_{0111}+N_{1011}+N_{1101}+N_{1110}&=N_0\binom{N_1}{3}\\
N_{1111}&=\binom{N_1}{4}.
\end{align*}
}

Among the four quantities $N_{0001},N_{0010},N_{0100},N_{1000}$ we have further relations
\begin{align*}
N_0N_{001}&=3N_{0001}+N_{0010}+2N_{001}\\
N_0N_{010}&=2N_{0010}+2N_{0100}+2N_{010}\\
N_0N_{100}&=3N_{1000}+N_{0100}+2N_{100}\\
\end{align*}
so there is at most one algebraically independent new one, which we can take for the moment to be $N_{1000}$.
Likewise among $N_{0111},N_{1011},N_{1101},N_{1110}$ there is (at most) one new algebraically independent one,
which we can take to be $N_{1110}$. 

For the six remaining quantities $N_{0011},\dots,N_{1100}$, we have relations
\begin{align*}
N_0N_{011}&=2N_{0011}+N_{0101}+N_{0110}+N_{011}\\
N_0N_{101}&=N_{0101}+2N_{1001}+N_{1010}+N_{101}\\
N_0N_{110}&=N_{0110}+N_{1010}+2N_{1100}+N_{110}
\end{align*}
and three more switching $0$ and $1$. 
Furthermore
$$N_{10}^2=2N_{1010}+4N_{1100}+2N_{110}+2N_{100}+N_{10}.$$
Combining these we find
$$N_{1001}=N_1N_{100}+N_{110}-\binom{N_{10}}2,$$
which is a surprising case where a length-$4$ pattern is an algebraic function of patterns of length strictly less than $4$.
In all we find, modulo patterns of length $\le3$, at most three algebraically independent length $4$ patterns,
which we take to be $N_{1000},N_{1100},N_{1110}$.

We now claim that $N_1,N_{10},N_{100},N_{110},N_{1000},N_{1100},N_{1110}$ are in fact algebraically independent.
This can be seen by considering sequences $X=1^{a_1}0^{a_2}\dots1^{a_7}0^{a_8}$ for varying
$a_1,\dots,a_8$. Each of the above $N_\tau(X)$'s is a polynomial function of the $a_i$,
and (an easy computer algebra calculation shows that)
the Jacobian matrix of the vector of $N_\tau$'s as a function of the $a_i$'s is full rank
for generic $a_i$. 

For sequences of length $5$ a similar calculation shows that there are $6$ new algebraically independent patterns, and we conjecture that 
for sequences of length $k$ there are $\binom{k-1}{[(k-1)/2]}$ new algebraically independent patterns,
so that for sequences of length up to $k$ the number of algebraically independent patterns $A(k)$ is conjecturally
$$A(k) = \sum_{j=1}^{k} \binom{j-1}{[\frac{j-1}2]}.$$

\section{Feasibility regions}\label{feasible}

For a given set of words $W=\{w_1,\dots,w_k\}$, typically not all densities $(\delta_1,\dots,\delta_k)\in[0,1]^k$ are feasible,
that is, the set of sublebesgue measures achieving those densities of these words may be empty.
We set $E_W\subset[0,1]^k$ to be the set of \emph{feasible constraints}:
$$E_W = \{(\delta_1,\dots,\delta_k)\in[0,1]^k~|~\exists\mu=f(x)dx\in\Omega:
(\rho_{w_1}(f),\dots,\rho_{w_k}(f)) =(\delta_1,\dots,\delta_k)\}.$$
 
As an example,
for a given $\delta_1=\rho_1(f)$, the density $\delta_{10}=\rho_{10}(f)$ 
is maximized at value $2\delta_1(1-\delta_1)$ when $f$ is the function 
$$f(x)=\begin{cases}1&x\le \rho_1\\0&\text{else}.\end{cases},$$
and minimized at value $0$ when $$f(x)=\begin{cases}0&x\le 1-\rho_1\\1&\text{else},\end{cases}$$ 
and can take any intermediate value (by taking a convex combination of these).
Thus
$$E_{\{1,10\}}=\{(\delta_1,\delta_{10})~:~0\le\delta_1\le1,~~~0\le\delta_{10}\le 2\delta_1(1-\delta_1)\}.$$ 
See Figure \ref{E1-10}.
\begin{figure}[htbp]
\begin{center}\includegraphics[width=2.5in]{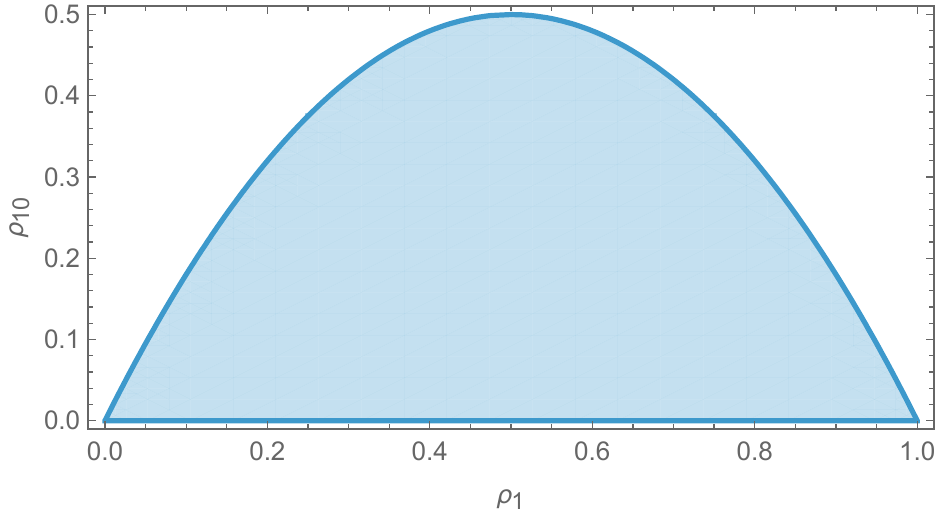}
\end{center}
\caption{\label{E1-10}Feasibility region for patterns $1$ and $10$.}
\end{figure}

While it is in general difficult to determine $E_W$, 
one easier case is when $W=\{1,w\}$, that is, when $W$ consists in two words one of which is $1$ (or $0$). 
See Theorem \ref{1k0thm} for another general family where $E_W$ is explicit.

\old{For example if $W=\{1,1011\}$ then for fixed $\delta_1=\rho_1(f)$, $\rho_{1011}$
is minimized at $0$ and maximized (see below) when $f$ is a step function with three steps,
one for each constant substring of $w$:
$$f(x) = \begin{cases}1&x< a\\0&a<x<b\\1&b<x\end{cases}$$
where $a+1-b=\delta_1$ and $a(b-a)(1-b)^2/2$ is maximal for $0\le a\le b\le 1$, that is $a=\frac{\delta}3, b=1-\frac23\delta$. Thus $E_W$ is the region defined by the inequalities $0\le\delta_{1011}\le\frac{16}9\delta_1^3(1-\delta_1)$.
}

\begin{thm}\label{Cthm} For fixed $\rho_1=\rho$ the set of feasibility for density $\rho_\tau$ is the interval 
$[0,C_\tau\rho^m(1-\rho)^n]$,
where $m,n$ are the number of $1$s and $0$s in $\tau$, respectively, and $C_\tau$ is a constant independent of $\rho$.
\end{thm}

\begin{proof} See (\ref{const}) in the proof of Theorem \ref{1010thm} below, which easily generalizes to this case.
\end{proof}

The value of $C_\tau$ can be quite nontrivial to calculate.
See the table below for $C_\tau$ for some choices of $\tau$. The proofs of these are sketched in the appendix.
\begin{center}\begin{tabular}{|r| l|}
\hline
pattern&$C$-value\\
\hline
$10$ &$2$\\
$1^k0^l$ & $\binom{k+l}{k}$\\
$1^{k}0^{l}1^{m}$ &$\binom{k+l+m}{k, l, m}\frac{k^km^m}{(k+m)^{k+m}}$\\
$1010$&$\frac{12}{e^2}$\\ 
$11010$&$30e^{-\pi/\sqrt{3}}$\\
$10110$&$\frac{20}9$\\
$10101$&$\frac{30\xi^2}{(1+\xi)^2}~~~\text{where $\xi e^\xi=e^{-1}$}$\\
\hline
\end{tabular}
\end{center}

As a special case, we analyze completely the case $\tau=1010$.
\begin{thm}\label{1010thm} For fixed $\rho_1=\rho$ the density $\rho_{1010}$ is maximized at $\rho_{1010}=\frac{12}{e^2}\rho^2(1-\rho)^2$,
and the unique maximizing sublebesgue measure has density
$$f(x)= \begin{cases}1&x<\frac{\rho}e\\
\frac12(1+\frac{e^{1/2}(x+\rho-1)}{\sqrt{4\rho(1-\rho)+e(x+\rho-1)^2}})&\frac{\rho}e<x<1-\frac{1-\rho}e\\0&1-\frac{1-\rho}e<x.\end{cases}$$

\end{thm}

See Figures \ref{f1010} and \ref{F1010YD} for the graph of $f$. 
Note that the choice of $\rho_1$ which maximizes the maximum $\rho_{1010}$-density $\frac{12}{e^2}\rho_1^2(1-\rho_1)^2$ is $\rho_1=\frac12$, and then $\rho_{1010}=\frac{3}{4e^2}$.

\begin{figure}[htbp]
\begin{center}\includegraphics[width=2in]{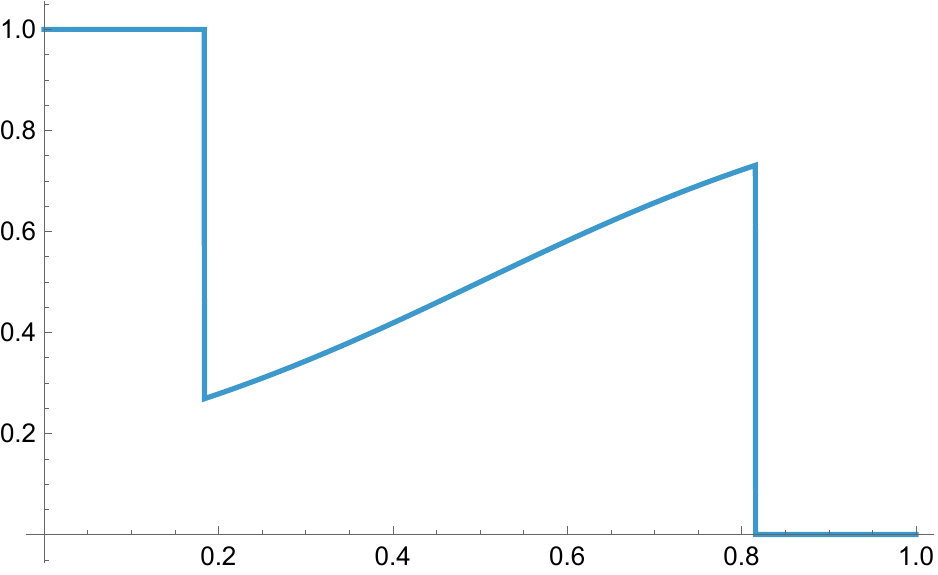}\hskip1cm\includegraphics[width=2in]{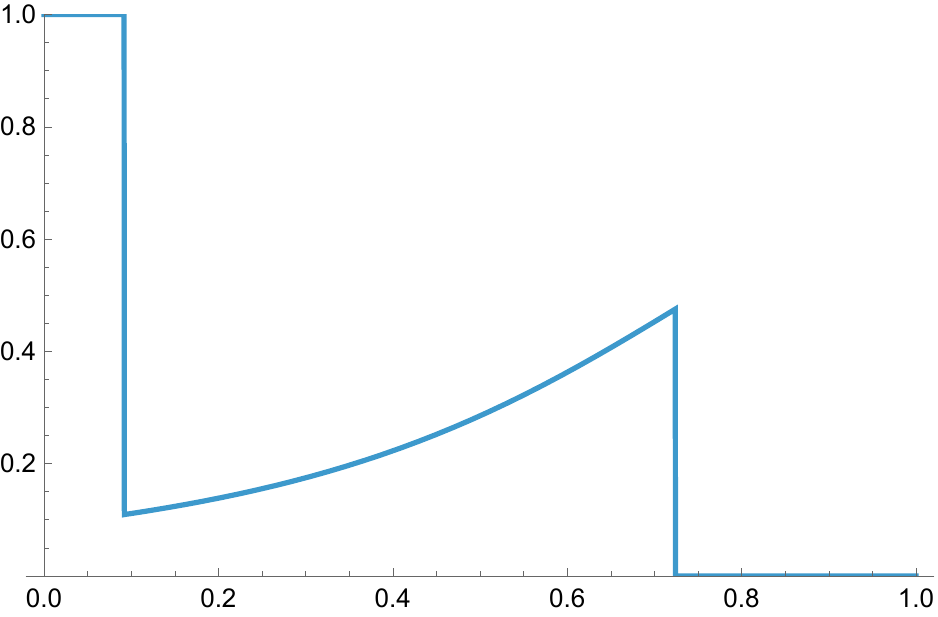}
\end{center}
\caption{\label{f1010}Plot of density $f(x)$ of the $\rho_{1010}$-maximizer when $\rho_1=\frac12$ and when $\rho_1=\frac14$.}
\end{figure}
\begin{figure}[htbp]
\begin{center}\includegraphics[width=1.5in]{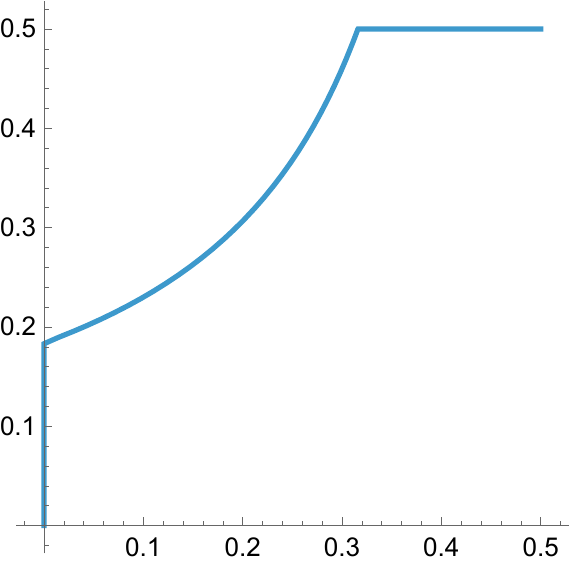}
\end{center}
\caption{\label{F1010YD}Plot of the $\rho_{1010}$-maximizer as a NE path, when $\rho_1=\frac12$.}
\end{figure}

\begin{proof}
We want to maximize $\rho_{1010}$ for fixed $\rho=\rho_1$. 
From (\ref{genpattern}) we have 
$$\rho_{1010} = 24\int_{0<x<y<z<w<1}f(x)(1-f(y))f(z)(1-f(w))dx\,dy\,dz\,dw.$$
Let $F(x) = \int_0^x f(y)dy$ be the distribution function and $H=F^{-1}$ the inverse distribution function. 
Since $F$ is nondecreasing,
$H'$ is defined as a generalized function ($H$ may have jump singularities), and $H'(y)\ge1$ since $F'(x)\le 1$. We can rewrite the integral
for $\rho_{1010}$ using the change of variable $u=F(x),~ x=H(u)$ where $u$ runs from $0$ to $\rho$ as $x$ runs from $0$ to $1$.
Thus 
\begin{align*}\rho_{1010} &= 24\int_{x<y<z<w}f(x)(1-f(y))f(z)(1-f(w))dx\,dy\,dz\,dw\\
&=24\int_{0<p<u<q<v<\rho} (H'(u)-1)(H'(v)-1)dp\,du\,dq\,dv\\
&=24\int_{0<u<v<\rho} u(v-u)(H'(u)-1)(H'(v)-1)\,du\,dv.
\end{align*}
Let $h(u) := H'(u)-1\ge0$ giving
$$\rho_{1010}=24\int_0^{\rho}\int_0^v u(v-u)h(u)h(v)\,du\,dv$$
which we need to maximize over nonnegative (generalized) functions $h$, where $h$ integrates to $1-\rho$ on the interval $[0,\rho]$.
Note the scale invariance: by scaling $u,v$ by $\rho$ we can change the domain of definition of $h$ to $[0,1]$.
We can then scale vertically by $\frac{\rho}{1-\rho}$ to make the integral of $h$ equal to $1$.
In other words, letting $g(x) := \frac{\rho}{1-\rho}h(x\rho)$ then $g$ is defined on $[0,1]$, nonnegative, integrates to $1$
and does not depend on $\rho$ at all.
Consequently 
\be\label{const}\rho_{1010}=C\rho^2(1-\rho)^2\ee
where $C$ does not depend on $\rho$:
\be\label{Cdef}C:=24\max_{g}\int_{0<x<y<1}x(y-x)g(x)g(y)\,dx\,dy\ee 
and the max is over nonnegative generalized functions $g\ge0$ with integral $1$, that is, over probability measures on $[0,1]$.
(The probabilistic interpretation of (\ref{Cdef}) is as follows: take two independent random samples from $g$; let $x$ be the smaller and $y$ be the larger. Which distribution $g$ maximizes the expected value of $x(y-x)$?)

To find the maximum, we consider perturbations $g\to g+\eps k$ where $k$ is of integral zero. Letting $C(g)$ denote the integral on the RHS of (\ref{Cdef}), we have (dropping the scalar $24$)
\be\label{dde}0=\frac{d}{d\eps}C(g+\eps k)\Big|_{\eps=0} = \int_{0<x<y<1}x(y-x)(g(x)k(y)+k(x)g(y))dx\,dy.
\ee
If $a$ is a point where $g(a)>0$, we can choose $k(u)$ to approximate the delta function derivative $\delta'_a(u)$ and this becomes after integration by parts
\begin{align*}0&= \int_{0}^axg(x)dx+\int_{a}^{1}(y-2a)g(y)dy\\
&=\int_{0}^{1}xg(x)\,dx-2a(G(1)-G(a))
\end{align*}
where $G(a)=\int_0^ag(x)\,dx$ is the distribution function of $g$.
Letting $2c=\int_{0}^{1}xg(x)\,dx$ this shows that $G(a) = G(1)-\frac{c}{a},$
so $g(a) = \frac{c}{a^2}$ when $g>0$. 

We also have to consider the possibility that $g\equiv0$ on an interval (since in this case perturbing by 
$\delta'_a(u)$ is not possible because $g$ is constrained to be nonnegative). Suppose $g$ is positive at a point $a$:  $g(a)>0$,  and a point $b$ to its right,
$b>a$, is contained in an interval on which
$g\equiv0$. Then consider a perturbation $g\mapsto g+\eps k$ with $k$ approximating $-\delta_a + \delta_b$.  Equation (\ref{dde}) becomes
$$0=(b-a)\int_0^1x g(x)dx+\int_a^b(-x^2+a^2)g(x)dx+(a^2-b^2)\int_b^1g(y)dy.$$
Consider the RHS as a function of $b$. 
Here the middle integral is independent of $b$ near $b$ since $g$ is zero there; but then we have
$$0=2c(b-a)+c_2 + (a^2-b^2)(G(1)-G(b))$$ which implies that $G(b)$ is a nonconstant function of $b$ near $b$, contradicting
the fact that $g$ is zero near $b$. Thus $g$ cannot be zero on such an interval.

It is however possible that $g\equiv0$ on an interval with no positive $g$-values to its left, that is, on an interval containing $0$. 
By symmetry (switching $0$'s and $1$'s in the pattern) 
if $g(x)=0$ for $x\in[0,b]$ then $g$ can also have a point mass $b'\delta_1$, of some mass $b'$, at $1$.
This leads us to consider $g$ of the form 
$$g(x) = \begin{cases}0&x<b\\ \frac{c}{x^2}&b<x<1\\
b'\delta_{1}&x=1.\end{cases}$$
The conditions $\int_0^1g(x)dx=1$ and $2c=\int_0^1xg(x)\,dx$ determine $b',c$ as functions of $b$. Then maximizing over $b$
yields $b=b'=c=e^{-1}$. For these values we find $C = \frac{12}{e^2}$.
Solving for $f$ yields the formula in the statement. 
\end{proof}

\section{Variational principle}\label{variational}

We now take $n$ large and consider sequences in $\{0,1\}^n$ with fixed densities $\rho_{w_i}$ of certain patterns $w_1,\dots,w_k$.
What does a typical such sequence look like?

We can study the limits, as $n\to\infty$, of such constrained random sequences by a variational principle,
maximizing the entropy of a sublebesgue measure subject to those density constraints.
For a sublebesgue measure $\mu=f(x)\,dx$ we define the entropy $\ENT(f)$ by
$$\ENT(\mu) = \int_0^1 S(f(x))\,dx$$
(where $S(p) = -p\log p-(1-p)\log(1-p)$ is the Shannon entropy function).
The entropy $\ENT(f)$ is the exponential growth rate of sequences whose density lies near $f$, 
in the following sense (see \cite{Vershik}).
Let $Z_n=\{X\in\{0,1\}^n~|~d_W(\mu_X,\mu)\le\eps\}$. Then
$$\lim_{\eps\to0}\lim_{n\to\infty}\frac1n\log|Z_n| = \int_0^1S(f(x))\,dx.$$

For a particular set of constraints $\rho_{w_1}=\delta_1,\dots,\rho_{w_k}=\delta_k$, the limiting typical density can be obtained
by maximizing
$$\E(f) = \ENT(f) + a_1\rho_{w_1} +\dots +  a_k\rho_{w_k}$$
where the $a_i$ are Lagrange multipliers determined by the constraints. We call an optimizer a \emph{limit shape}.
We don't have at present any guarantee of uniqueness of the optimizer, although we don't know of any cases 
where it is not unique.

\section{$0, 10, 110,\dots,1^k0$ patterns}\label{1k0}

Let us first find the optimizing sublebesgue measure (limit shape) with fixed pattern densities $\rho_0,\rho_{10}$ and $\rho_{110}$.
We discuss the general case afterwards.
Using the change of variables $y=F(x), x=H(y)$ with $dx=H'(y)dy$ we have
\begin{align*}\rho_{110} &= 6\int_{0\le x<u<v\le 1}f(x)f(u)(1-f(v))\,dx\,du\,dv\\
& = 6\int_{0\le a<b<y\le\rho_1} (H'(y)-1)da\,db\,dy \\
&= 3\int_0^{\rho_1} y^2(H'(y)-1)\,dy.
\end{align*}

To maximize $f$ for fixed $\rho_0,\rho_{10}$ and $\rho_{110}$,
the variational equation is
$$\max_f\left\{\int_0^1 S(f(x))\,dx + A\rho_{0}+B\rho_{10}+C\rho_{110}\right\}$$
where $A,B,C$ are Lagrange multipliers.
Using the above change of variables this is
$$\max_H\left\{\int_0^{\rho_1}S(1/H'(y))H'(y)\, dy +\int_0^{\rho_1}(A+2By+3Cy^2)(H'(y)-1)\,dy \right\}.$$
Let $h(y)=H'(y)-1$.
Replacing $h(y)$ with $h(y)+\eps \delta'_{y=z}$, and taking the derivative with respect to $\eps$ at $\eps=0$ yields
$$0=\int_0^{\rho_1}-\log\frac{h(y)}{h(y)+1}\delta'_{y=z}\, dy +\int_0^{\rho_1}(A+2By+3Cy^2)\delta'_{y=z}\,dy$$
$$0=-\frac{h'(z)}{h(z)^2+h(z)} + 2B+6Cz.$$
This has the solution
\be\label{lsd}H'(y)=h(y)+1 = \frac1{1-e^{a+by+cy^2}}\ee
for constants $a,b,c$. 
Here $b=2B,c=3C$ depend on $B,C$ and $a$ must be chosen so that $H(\rho_1)=1$, as follows. 
For fixed $\rho_1$, the integral $\int_0^{\rho_1}H'(y)dy$ is not defined if $a+by+cy^2$ has a zero in $[0,\rho_1]$,
and the integrand is negative if $a+by+cy^2>0$ for $y\in[0,\rho_1]$. So we need to choose $a<0$ so that
$a+by+cy^2<0$ for $y\in[0,\rho_1]$. Moreover under this constraint the integral is a monotone function of $a$,
so there is a unique value of $a<0$ for which $1=H(\rho_1)=\int_0^{\rho_1}H'(y)dy$. 

Conversely,
given $a,b,c\in\R$ with $a<0$ we can define $\rho_1$ to be the first positive value for which 
$\int_0^{\rho_1}H'(y)dy=1$; necessarily $0<\rho_1<1$ since $H'(y)>1$. Then $\rho_0=1-\rho_1$, and
$\rho_{10},\rho_{110}$ are defined by the integrals
$$\rho_{10}=2\int_0^{\rho_1}y(H'(y)-1)\,dy,~~~~~\rho_{110}=6\int_0^{\rho_1}y^2(H'(y)-1)\,dy.$$

We thus see that the limit shape density is of the form (\ref{lsd}), where $a,b,c$ are functions of the Lagrange multipliers 
$A,B,C$. Moreover the map from triples $(a,b,c)$ with $a<0$ to triples $(A,B,C)$ is a homeomorphism; see below.

The same calculation above applies to any set of patterns $\rho_0,\rho_{10},\rho_{110},\dots,\rho_{1^k0}$:

\begin{thm}\label{1k0thm} For a set of pattern densities $(\rho_0,\rho_{10},\rho_{110},\dots,\rho_{1^k0})=(\delta_0,\dots,\delta_k)$
the resulting entropy-maximizing function $H'(y)$ is of the form
\be\label{Hpform}H'(y) = \frac1{1-e^{p(y)}}\ee
for a polynomial $p$ of degree $k$ with real coefficients and negative constant coefficient. The map $\Phi:\R_-\times\R^{k-1}\to E_W$
from such polynomials to 
the interior of $E_W$, the feasible region, is a homeomorphism. Densities on the boundary of $E_W$ are $\{0,1\}$-valued step functions having at most $k/2+1$ intervals on which they take value $0$ (and if this number is equal to $k/2+1$, the first and last steps are intervals of value $0$).
\end{thm}

For example suppose we want to fix (only) $\rho_1$ and $\rho_{10}$.
Then we need to consider $H$ of the form
$$H'(y)=\frac1{1-e^{a+by}}$$
for constants $a<0$ and $b\in\R$. This leads to densities $f$ of the form $f(x) = \frac1{1+c e^{bx}}$ where $c=e^a/(1-e^a)$,
reconstructing the result of \cite{Vershik}. The upper boundary $\rho_{10}=2\rho_0(1-\rho_0)$ of $E_{0,10}$ consists of step function densities
$f(x) = \begin{cases}1&x<\rho_1\\0&\text{else}\end{cases}$, obtained from the limit $a+bx=b(x-\rho_1)$ as $b\to\infty$ and the lower boundary consists of step function densities
$f(x) = \begin{cases}0&x<\rho_0\\1&\text{else}\end{cases},$ obtained from the limit $a+bx=b(x-\rho_0)$ as $b\to-\infty$.
See Figure \ref{r110fixed} for another example.
\begin{figure}[htbp]
\begin{center}
\includegraphics[width=6cm]{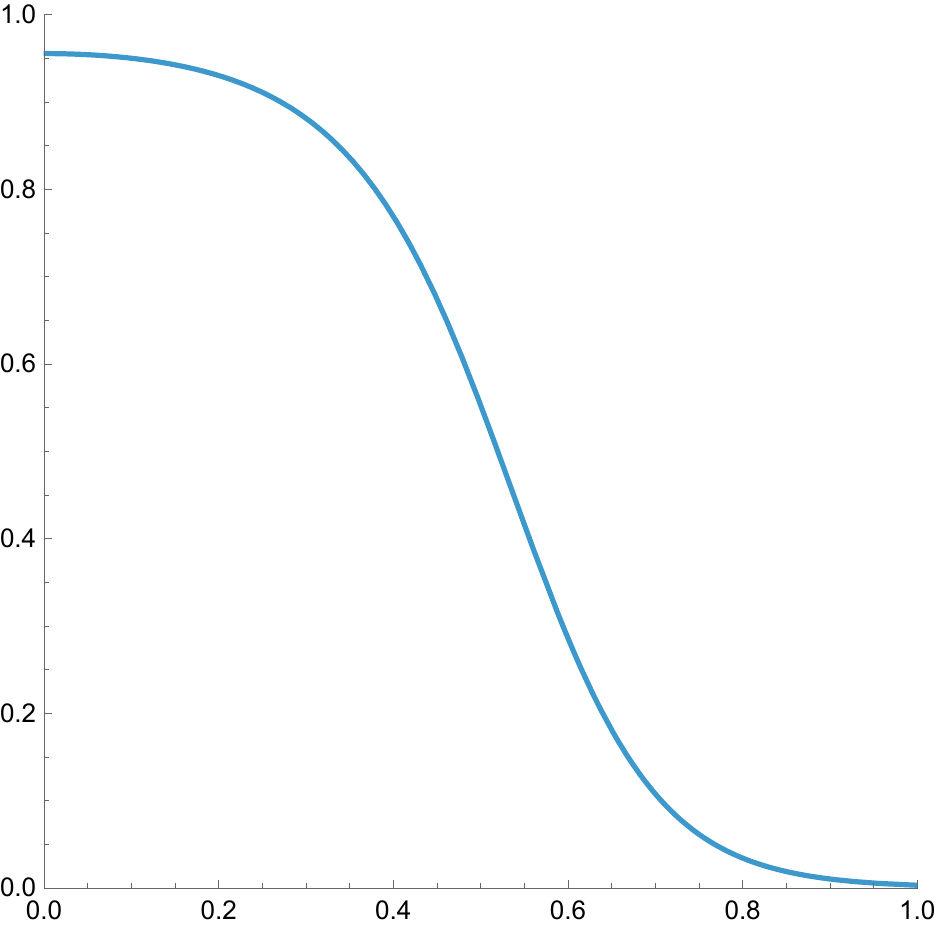}
\end{center}
\caption{\label{r110fixed}The entropy maximizing density $f(x)$ for fixed $\rho_1=\frac12$ and $\rho_{110}=\frac13$.
In this case $H'(y) = \frac1{1-e^{-a+by^2}}$ with $a\approx 3.10795$ and $b\approx 12.42$.}

\end{figure}
\begin{proof}
The fact that $H'(y)$ has the form (\ref{Hpform}) follows from the remarks preceding the statement.
Let $\Phi$ be the map from polynomials $p$ (of degree $\le k$, with negative constant term) to $(\rho_0,\dots,\rho_k)$. It remains to show that $\Phi$ is a proper local homeomorphism. 

To see that $\Phi$ is a local homeomorphism, we show that
the Jacobian of $\Phi$
is nonzero: $J(\Phi)\ne 0$. 
To simplify notation let $\rho_i:=\rho_{1^i0}$ and $\rho=\rho_1$. Let $p(y) = a_0+a_1y+\dots+a_ky^k$. 
First, differentiating the equation $1=\int_0^\rho \frac{1}{1-e^{p(y)}}\,dy$ with respect to $a_j$ gives
\be\label{drdai}\frac1{1-e^{p(\rho)}}\frac{d\rho}{da_j} =-\int_0^\rho \frac{y^j e^{p(y)}}{(1-e^{p(y)})^2}.\ee

We now compute 
\begin{align}\nonumber\frac{d\rho_i}{da_j} &= \frac{d}{da_j}\int_0^\rho \frac{y^ie^{p(y)}}{1-e^{p(y)}}\,dy\\
\nonumber&=\int_0^\rho \frac{y^{i+j}e^{p(y)}}{(1-e^{p(y)})^2}\,dy + \frac{\rho^ie^{p(\rho)}}{1-e^{p(\rho)}}\frac{d\rho}{da_j}\\
\label{dridaj}&=\int_0^\rho \frac{y^{j}(y^i-\rho^ie^{p(\rho)})e^{p(y)}}{(1-e^{p(y)})^2}\,dy.
\end{align}
The differential of $\Phi$ is
$$D\Phi = \left(\frac{d\rho_i}{da_j}\right)_{i,j=0,\dots,k}.$$
Using linearity of the determinant over columns and (\ref{dridaj}) we can write
$$J\Phi = \int_0^\rho\dots\int_0^\rho d\mu(y_0)\dots d\mu(y_k)\det\begin{vmatrix}
1-c&y_1(1-c)&\dots&y_k^k(1-c)\\
y_0-\rho c&y_1(y_1-\rho c)&\dots&y_k^k(y_k-\rho c)\\
\dots\\
y_0^k-\rho^k c&y_1(y_1^k-\rho^kc)&\dots&y_k^k(y_k^k-\rho^kc)
\end{vmatrix}.$$
where $c=e^{p(\rho)}$ and $d\mu(y) = \frac{e^{p(y)}}{(1-e^{p(y)})^2}\,dy.$
Factoring $1-c$ out of row $1$, factoring $y_i^i$ out of column $i+1$ and doing row operations
leaves a Vandermonde matrix, giving
$$J\Phi = (1-c)\int\dots\int d\mu(y_0)\dots d\mu(y_k)y_1y_2^2\dots y_k^k\prod_{i<j}(y_j-y_i).$$

By Lemma \ref{vdm} below, $J\Phi>0$. This proves that $\Phi$ is a local homeomorphism.

We now show that $\Phi$ is proper: when $a_0\to0$, or any $a_i\to\pm\infty$, the associated density tends to a point on the boundary of $E_W$. First, if $a_0$ tends to zero, $\rho_1\to0$. Secondly, if any coefficient(s) $a_i$ tend to $\pm\infty$
(while $p$ remains negative on $[0,\rho_1]$)
then $p$ tends to $-\infty$ at all but at most $k$ points: Lagrange interpolation determines a polynomial of degree $k$ uniquely 
by its values at any $k+1$ points. In fact there are at most $k/2+1$ points where $p$ remains bounded, since these points are necessarily either local maxima of $p$ or the endpoints $p(0)$ or $p(\rho_1)$.
In such a limit $H'$ is supported at at most $k/2+1$ points, and could have a $\delta$-function there if the corresponding $p$ values 
tend to zero. 
These delta functions correspond to intervals on which the limiting $f$ takes value $0$; outside of these $H'=1$ so the limiting
$f$ necessarily takes value $1$. 
\end{proof}

\begin{lemma}\label{vdm}
For any measure $\mu$ on an interval $[a,b]$, where $\mu$ is not supported on fewer than $k+1$ points, we have
$$\int_a^b\dots\int_a^b y_1y_2^2\dots y_k^k\prod_{i<j}(y_j-y_i)d\mu(y_0)\dots d\mu(y_k)>0.$$
\end{lemma}

\begin{proof} Sum over all permutations of $y_0,\dots,y_k$ to give
$$\int_a^b\dots\int_a^b d\mu(y_0)\dots d\mu(y_k)\left(\sum_{\sigma} (-1)^{\sigma}y_{\sigma(0)}^0y_{\sigma(1)}^1\dots y_{\sigma(k)}^k\right)\prod_{i<j}(y_j-y_i).$$
The factor in parentheses is another copy of $\prod_{i<j}(y_j-y_i)$, so we get
$$\int_a^b\dots\int_a^b d\mu(y_0)\dots d\mu(y_k)\prod_{i<j}(y_j-y_i)^2$$
which is manifestly positive.
\end{proof}

\section{Upper-triangular matrices}\label{positivity}

The integer Heisenberg group $H$ is the group of upper triangular $3\times 3$ matrices 
$$M=\begin{pmatrix}1&a&b\\0&1&c\\0&0&1\end{pmatrix}$$ with $a,b,c\in\Z.$ 
It is generated by matrices 
$$M_0=\begin{pmatrix}1&1&0\\0&1&0\\0&0&1\end{pmatrix},~~~M_1=\begin{pmatrix}1&0&0\\0&1&1\\0&0&1\end{pmatrix}.$$
Given a binary sequence $X$ there is a corresponding matrix $M_X$ which is a positive word in the generators,
for example $M_{01101}=M_0M_1M_1M_0M_1=\begin{pmatrix}1&2&4\\0&1&3\\0&0&1\end{pmatrix}.$
The superdiagonal entries of $M_X$ are the number of $0$s and $1$s,
respectively, in $X$. 
The upper-right entry is the number of $01$-patterns in $X$. 

More generally we can take a generalized Heisenberg group generated by larger matrices, where $M_0$ has certain superdiagonal entries $1$
and $M_1$ has the complementary superdiagonal entries equal to $1$. For example if we take
$$M_0=\begin{pmatrix}1&1&0&0\\0&1&0&0\\0&0&1&1\\0&0&0&1\end{pmatrix},~~~M_1=\begin{pmatrix}1&0&0&0\\0&1&1&0\\0&0&1&0\\0&0&0&1\end{pmatrix},$$
then in this case 
\be\label{MX}M_X=\begin{pmatrix}1&N_0&N_{01}&N_{010}\\0&1&N_1&N_{10}\\0&0&1&N_0\\0&0&0&1\end{pmatrix},\ee
where $N_\tau=N_\tau(X)$.
In general, if the superdiagonal entries of $M_1$ are $a_1,a_2,\dots,$ where $a_i\in\{0,1\}$ then the entries on the first row
of $M_X$ are $1, N_{a_1}, N_{a_1a_2}, N_{a_1a_2a_3},\dots.$
Thus any pattern-counting function $N_w(X)$ appears as the upper-right entry in some such matrix product.

The optimal pattern-counting problem for a pattern $w$  can be phrased as finding the set of shortest geodesics in the appropriate Heisenberg group from the identity to the subset of matrices whose upper right entry is $N_w$. (For 
a related result over the field ${\mathbb F}_p$, see \cite{DH}.)
 
The matrices $M_X$ are totally nonnegative: all minors are nonnegative. 
One way to see this is through the LGV lemma \cite{Lindstrom, GV},
see Figure \ref{LGV}. 
\begin{figure}[htbp]
\begin{center}
\includegraphics[width=8cm]{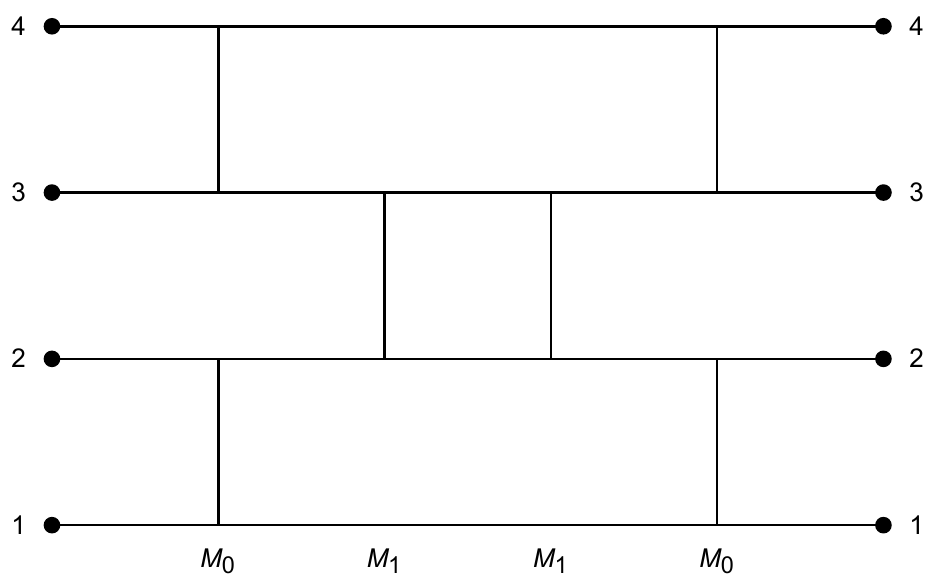}
\caption{\label{LGV}A non-intersecting path model for the determinant of a minor of $M_{0110}$.
Note that the number of NE paths from the vertex labeled $1$ on the left to the vertices labeled $2,3,4$ on the right are, respectively,
$N_0,N_{01},N_{010}$. From vertex $2$ on the left to vertices $2,3,4$ on the right, the numbers are $1,N_1,N_{10}$ and from vertex $3$ the number are
$0,1,N_0$. So triples of nonintersecting NE paths from $1,2,3$ on the left to $2,3,4$ on the right are counted by the 
upper right $3X3$ minor of (\ref{MX}) when $X=0110$.}
\end{center}
\end{figure} 

This positivity implies certain restrictions on $E_W$ for general sets of patterns $W$. 
However at the moment we have not been able to usefully apply this idea to compute $E_W$ for interesting
sets $W$. 

\section{Open questions}

\begin{enumerate}
\item If, instead of binary sequences, we consider a $k$-letter alphabet, are there any new phenomena?

\item Is the entropy an analytic function on the interior of the feasibility set $E_W$? Is the entropy-maximizing
sublebesgue measure always unique on the interior of $E_W$?

\item Is the feasibility set $E_W$ always a semi-algebraic set for any $W$?
\end{enumerate}

\section{Appendix}
We sketch here derivations of the $C_\tau$ values mentioned in the table.

\subsection{$1^k0^l1^m$}
For $\tau=1^k0^l1^m$, to maximize $N_{\tau}$ for fixed $N_1$, it is optimal to have all the $0$s be consecutive;
the limiting density 
$f$ must therefore be a step function of the form
$$f(x) = \begin{cases}1&x<a\\0&a<x<b\\1&b<x<1\end{cases}$$
for some $a$ where $b-a=1-\rho_1$. The density $\rho_\tau$ is then
$$\rho_\tau=(k+l+m)!\frac{a^k}{k!}\frac{(1-\rho_1)^l}{l!}\frac{(\rho_1-a)^m}{m!}.$$ This is maximized when $a=\frac{k\rho_1}{k+m}$.
Plugging in gives the result.

\subsection{$11010$}
For $\tau=11010$, using the method of proof of Theorem \ref{1010thm} leads to 
$$C_\tau:=5!\max_g\int_{0<x<y<1}x^2(y-x)g(x)g(y)\,dx\,dy.$$
As in the $1010$ case we look for $g$ being $0$ on an interval $[0,b]$, 
having a point mass at $1$, and analytic in between.
Considering perturbations $g\to g+\eps\delta'_a$ gives Euler-Lagrange equation
\be\label{EL11010}0=\int_0^a x^2 g(x)\,dx + \int_a^1a(2y-3a)g(y)\,dy.\ee
Differentiating three times with respect to $a$ yields
\be\label{d2}2a^2 g''(a)+12a g'(a)+14 g(a)=0.\ee  
This yields the form of the analytic part of $g$, 
$$g(a) = c a^{-5/2}\cos(\frac{\pi}3+\frac{\sqrt{3}}2\log a)$$ for a constant $c$ (where we eliminated one constant of integration using
the \emph{second} derivative of (\ref{EL11010})).
The remainder of the calculation is similar to that 
in the proof of Theorem \ref{1010thm}. 

\subsection{$10110$} As in the $11010$ case except the Euler-Lagrange equation is 
$$0=2\int_0^a x(a-x) g(x)\,dx + \int_a^1(y-a)(y-3a)g(y)\,dy.$$
Differentiating $3$ times with respect to $a$ leads to $g(a)=0$. In this case however we can and will 
have a point mass in the interior, so we expect $g=c_1\delta_b+c_2\delta_1$. This implies that $f$ is a $\{0,1\}$-valued step
function, with $4$ steps.

\subsection{$10101$}
As in the $11010$ case except the Euler-Lagrange equation is 
$$0=2\int_0^a x(1-2a+x) g(x)\,dx + \int_a^1(1-y)(y-2a)g(y)\,dy.$$
Differentiating twice with respect to $a$ leads to $2a(1-a)g'(a)+(4-8a)g(a)=0$,
with solution $g(a) = \frac{c}{a^2(1-a)^2}$ for a constant $c$.
As in the $1010$ case we can also expect $g$ to be zero on an interval $[0,b]$ and, symmetrically, $[1-b,1]$.
Thus we are led to 
$$g(x) = \begin{cases}0&0<x<b\\\frac{c}{x^2(1-x)^2}&b<x<1-b\\
0&1-b<x<1.\end{cases}$$
where $c=c(b)$ is determined by the fact that $\int_0^1g(x)dx=1$. 
Integrating and optimizing over $b$ leads to $b=\xi/(1+\xi)$ and the given value for $C_\tau$.

\old{

\section{A 3D domain}
Consider the domain $E_{1,110,011}$.
We have $\rho_{110}+\rho_{011}+\rho_{101}=3\rho_1^2(1-\rho_1),$ so for fixed $\rho_1$ and $\rho_{011}$, the value of $\rho_{110}$ is maximized when $\rho_{101}=0$, that is, when $f$ is a step function with values $0,1,0$, and 
$\rho_{110}\le 3\rho_1^2(1-\rho_1)-\rho_{011}.$
The value of $\rho_{110}$ is minimized when $f$ is a step function with values $1,0,1,$
giving $\rho_{110}\ge $
See Figure \ref{E1.110.011}.
Consequently the region $E_{110,011,101}$  is $$E_{110,011,101}=\{0\le\rho_{110}+\rho_{101}+\rho_{011}\le \frac49\}.$$
}

\bibliographystyle{plain}
\bibliography{seq}

\end{document}